\documentclass[12pt]{amsart}

\usepackage[a4paper,margin=3cm]{geometry}
\usepackage{amsmath,amssymb,amsthm}
\usepackage{stmaryrd}
\usepackage[T1]{fontenc}
\usepackage[utf8]{inputenc}
\usepackage{lmodern}
\usepackage{microtype}
\usepackage[hidelinks]{hyperref}
\newcommand{\DD}{\mathbb D}
\newcommand{\CC}{\mathbb C}
\newcommand{\RR}{\mathbb R}
\newcommand{\GG}{\mathbb G_2}
\newcommand{\cR}{\mathcal R}
\newcommand{\Aut}{\operatorname{Aut}}
\newcommand{\supp}{\operatorname{supp}}

\newtheorem{theorem}{Theorem}

\newtheorem{lemma}[theorem]{Lemma}

\begin{document}

\title{Isometries in the symmetrized bidisc, II}
\author{Armen Edigarian}
\address{Faculty of Mathematics and Computer Science, Jagiellonian University, ul. \L ojasiewicza 6, 30-348 Krak\'ow, Poland}
\email{armen.edigarian@uj.edu.pl}
\thanks{Funded by the National Science Centre, Poland under the Weave UNISONO, UMO-2025/07/Y/ST1/00146}

\begin{abstract}
We prove that every map $f:U\to\GG$ preserving the Poincar\'e distance and the Kobayashi distance is
holomorphic or anti-holomorphic, 
where $U$ is a connected open subset of the unit disc. 
We also study the nonroyal automorphism orbits of $\GG$. A $C^1$ map on a connected relatively open part of such an orbit which preserves the restriction of the
ambient Kobayashi--Royden metric is the restriction of a global automorphism or
anti-automorphism of $\GG$.  The same conclusion holds for maps preserving the
ambient Kobayashi distance.
\end{abstract}

\maketitle

\section{Introduction}

The symmetrized bidisc is
\[
 \GG=\{(z+w,zw):z,w\in\DD\}\subset\CC^2.
\]
Its Carath\'eodory and Kobayashi distances coincide, and the corresponding
Carath\'eodory--Reiffen and Kobayashi--Royden metrics coincide; see
\cite{AglerYoung2004,Costara2004}.  We write $k_{\GG}$ for the common distance
and $\kappa_{\GG}$ for the common infinitesimal metric.  The Poincar\'e distance
of $\DD$ is denoted by $\rho$, with infinitesimal metric
$|v|/(1-|\lambda|^2)$.

In \cite{Edigarian2024} it was proved that every $C^1$ infinitesimal isometry
from the disc to $\GG$ is holomorphic or anti-holomorphic.  The first result removes the regularity assumption.
\begin{theorem}\label{thm:disc}
Let $U\subset\DD$ be connected and open and let $f:U\to\GG$ satisfy
\[
 k_{\GG}(f(\lambda_1),f(\lambda_2))=\rho(\lambda_1,\lambda_2),\qquad \lambda_1,\lambda_2\in U.
\]
Then $f$ is holomorphic or anti-holomorphic on $U$.
\end{theorem}

The second part concerns the automorphism orbits.  The royal variety
$\cR=\{(2\lambda,\lambda^2):\lambda\in\DD\}$ is the automorphism orbit of the
origin.  Every other orbit is a real three-dimensional hypersurface.  For
$0<r<1$ we write $x_r=(0,r)$ and $\mathcal O_r=\Aut(\GG)\cdot x_r$.

\begin{theorem}\label{thm:orbit}
Let $0<r<1$, let $V\subset\mathcal O_r$ be nonempty, connected and relatively
open, and let $f:V\to\GG$. The following conditions are equivalent:
\begin{enumerate}
\item[(a)] $f$ is $C^1$ and
\[
 \kappa_{\GG}(f(z);df_zv)=\kappa_{\GG}(z;v),
 \qquad z\in V,\quad v\in T_z\mathcal O_r;
\]
\item[(b)]
\[
 k_{\GG}(f(z_1),f(z_2))=k_{\GG}(z_1,z_2),\qquad z_1,z_2\in V;
\]
\item[(c)] $f$ is the restriction to $V$ of an automorphism or an
anti-automorphism of $\GG$.
\end{enumerate}
\end{theorem}

Here and below an anti-automorphism means $\sigma\circ\gamma$, where
$\gamma\in\Aut(\GG)$ and $\sigma(s,p)=(\overline s,\overline p)$.

\section{Basic formulas}

Every automorphism $m$ of $\DD$ induces an automorphism of $\GG$ by
\[
 \gamma_m(z+w,zw)=(m(z)+m(w),m(z)m(w)),
\]
and these are all automorphisms of $\GG$; see \cite{JarnickiPflug2004}.
If $(s,p)=(z+w,zw)\in\GG$, put
\[
 \Delta(s,p)=\left|\frac{z-w}{1-\overline w z}\right|\in[0,1).
\]
It is invariant under
$\Aut(\GG)$. Moreover, the $\Aut(\GG)$-orbits are exactly the level sets of $\Delta$.

For $0\le r<1$ put $x_r=(0,r)$ and $\mathcal O_r=\Aut(\GG)\cdot x_r$. So $\Delta(x_r)=\frac{2\sqrt r}{1+r}$.
This is strictly increasing in $r$.  Thus $\mathcal O_0=\cR$, the sets $\mathcal O_r$,
$0<r<1$, are pairwise disjoint nonroyal orbits, and every point of $\GG$ belongs
to exactly one of them.  
Every $\mathcal O_r$ is relatively closed in $\GG$.
A direct calculation from $s=z+w$ and $p=zw$ gives
\begin{equation}\label{eq:Delta}
 \Delta(s,p)^2=
 \frac{2|s^2-4p|}{2+2|p|^2-|s|^2+|s^2-4p|}.
\end{equation}

For $\omega\in\mathbb T$ let
\[
 \Phi_\omega(s,p)=\frac{2\omega p-s}{2-\omega s}.
\]
The functions $\Phi_\omega$ are Carath\'eodory extremals, and since the
Carath\'eodory--Reiffen and Kobayashi--Royden metrics agree on $\GG$, one obtains
at $x_r$ the formula
\begin{equation}\label{eq:kappa-max}
 \kappa_{\GG}(x_r;(a,b))
 =\frac1{1-r^2}\max_{|\omega|=1}
 \left|\omega b+\frac a2(r\omega^2-1)\right|.
\end{equation}
Equivalently, if $a\ne0$ and $\omega=e^{i\theta}$,
\begin{equation}\label{eq:ellipse}
 \kappa_{\GG}(x_r;(a,b))
 =\frac{|a|}{2(1-r^2)}\max_{\theta\in\RR}
 \left|\frac{2b}{a}-(1-r)\cos\theta+i(1+r)\sin\theta\right|.
\end{equation}
Thus, the norm is obtained by taking the farthest point on the ellipse.

We shall
also use the immediate special cases
\begin{equation}\label{eq:axes}
 \kappa_{\GG}(x_r;(a,0))=\frac{|a|}{2(1-r)},\quad a\in\CC,\qquad
 \kappa_{\GG}(x_r;(0,it))=\frac{|t|}{1-r^2},\quad t\in\RR.
\end{equation}

We have the following equality.
\begin{lemma}\label{lem:metric-derivative}
If $q_t=q+tX+o(t)$ in $\GG$ as $t\to0$ through real values, then
\[
 \lim_{t\to0}\frac{k_{\GG}(q_t,q)}{|t|}=\kappa_{\GG}(q;X).
\]
\end{lemma}

\begin{proof}
For every $F\in\mathcal O(\GG,\DD)$, contractivity of the Carath\'eodory
distance gives
$\rho(F(q_t),F(q))\le k_{\GG}(q_t,q)$.  Dividing by $|t|$, passing to the
limit inferior and taking the supremum over $F$ gives
$\kappa_{\GG}(q;X)=\gamma_{\GG}(q;X)\le\liminf k_{\GG}(q_t,q)/|t|$.
For the reverse inequality join $q$ to $q_t$ by the straight segment.  For
small $t$ it lies in a compact subset of $\GG$, and continuity of
$\kappa_{\GG}$ gives an integrated length
$|t|\kappa_{\GG}(q;X)+o(|t|)$.  Since $k_{\GG}$ is the integrated distance of
$\kappa_{\GG}$, the opposite inequality follows.
\end{proof}

We shall also use the following elementary local comparison.

\begin{lemma}\label{lem:comparison}
Let $D\subset\CC^n$ be bounded and $K\Subset D$ be a relatively compact set.  There are $a_K,b_K>0$ such that, for sufficiently close $z,w\in K$,
\[
 a_K\|z-w\|\le k_D(z,w)\le b_K\|z-w\|.
\]
\end{lemma}

\begin{proof}
Choose $R$ with $D\subset R\mathbb B^n$.  For $v\ne0$ the linear function
$z\mapsto\langle z,v/\|v\|\rangle/R$ maps $D$ to $\DD$, hence
$\kappa_D(z;v)\ge\|v\|/R$.  Integration gives the lower bound.  If
$\delta=\operatorname{dist}(K,\partial D)$, then for points sufficiently near $K$ the balls
of radius $\delta/4$ centered at those points lie in $D$.  Monotonicity of the
Kobayashi--Royden metric gives $\kappa_D(z;v)\le4\|v\|/\delta$ there, and
integration along the segment from $x$ to $y$ gives the upper bound.
\end{proof}

If a real differentiable map $f:\DD\to\CC^2$ is written as $f=(f_1,f_2)$, then 
$$
df_{\lambda}(v)=v f_z(\lambda)+\overline v f_{\bar z}(\lambda),\quad v\in\CC,
$$
where $f_z=(\frac{\partial f_1}{\partial z},\frac{\partial f_2}{\partial z})$ and
$f_{\bar z}=(\frac{\partial f_1}{\partial \bar z},\frac{\partial f_2}{\partial\bar z})$.

\section{Proof of Theorem~\ref{thm:disc}}

We divide the proof into three steps.

\medskip

\noindent
\textbf{Step 1.}
\emph{The map $f$ is locally Euclidean bi-Lipschitz and, for almost every
$\lambda\in U$, the differential $df_\lambda$ is either complex-linear or
conjugate-linear.}

\medskip

Injectivity follows immediately from distance preservation.
Lemma~\ref{lem:comparison}, together with the local equivalence of the
Poincar\'e and Euclidean distances, shows that $f$ is locally Euclidean
bi-Lipschitz. Hence, by Rademacher's theorem, $f$ is differentiable almost
everywhere.

Let $\lambda\in U$ be a differentiability point. By
Lemma~\ref{lem:metric-derivative}, for every $v\in\CC$,
\begin{equation}\label{eq:disc-inf}
 \kappa_{\GG}(f(\lambda);df_\lambda v)
 =
 \frac{|v|}{1-|\lambda|^2}.
\end{equation}
In particular, $df_\lambda$ is injective.

Suppose first that $f(\lambda)\notin\cR$. Since
\[
 df_\lambda(v)
 =
 v f_z(\lambda)+\overline v\,f_{\bar z}(\lambda),
 \qquad v\in\CC,
\]
for $|v|=1$, complex homogeneity of $\kappa_{\GG}$ and
\eqref{eq:disc-inf} give
\[
 \kappa_{\GG}\left(
 f(\lambda);
 f_z(\lambda)+\overline v^{\,2}f_{\bar z}(\lambda)
 \right)
 =
 \frac1{1-|\lambda|^2}.
\]
As $\overline v^{\,2}$ runs through $\mathbb T$, the pointwise rigidity
calculation from \cite{Edigarian2024} yields
\[
 f_z(\lambda)=0
 \qquad\text{or}\qquad
 f_{\bar z}(\lambda)=0.
\]

It remains to consider almost every point of $f^{-1}(\cR)$.
Let $\lambda\in f^{-1}(\cR)$ be both a differentiability point of $f$
and a Lebesgue density point of $f^{-1}(\cR)$. Put
\[
 \Psi(s,p)=s^2-4p,
\]
so that $\cR=\{\Psi=0\}$. Since $\Psi\circ f$ vanishes on a set of
density one at $\lambda$, differentiability gives
\[
 d(\Psi\circ f)_\lambda=0.
\]
Hence
\[
 df_\lambda(\CC)
 \subset
 \ker d\Psi_{f(\lambda)}
 =
 T_{f(\lambda)}\cR.
\]
Both sides have real dimension two, because $df_\lambda$ is injective.
Therefore
\[
 df_\lambda(\CC)=T_{f(\lambda)}\cR.
\]
Choose a nonzero complex tangent vector $e\in T_{f(\lambda)}\cR$ and
write
\[
 df_\lambda(v)=e\,\ell(v),
\]
where $\ell:\CC\to\CC$ is real linear. By complex homogeneity of
$\kappa_{\GG}$, equation~\eqref{eq:disc-inf} implies
\[
 |\ell(v)|=c|v|
\]
for some $c>0$. Thus $\ell$ is a similarity of the Euclidean plane,
and therefore is either complex-linear or conjugate-linear. The same
is consequently true of $df_\lambda$.

We have proved that, for almost every $\lambda\in U$,
$df_\lambda$ is either complex-linear or conjugate-linear.

\medskip

\noindent
\textbf{Step 2.}
\emph{The image $f(U)$ is locally contained in a one-dimensional complex
analytic set.}

\medskip

Fix $\lambda_0\in U$. Choose concentric discs
\[
 B'\Subset B\Subset U
\]
centered at $\lambda_0$ such that $f$ is bi-Lipschitz on a neighborhood
of $\overline B$. Since $f$ is injective on $\overline B$, we may choose
an open neighborhood $W$ of $f(\overline{B'})$ such that
\[
 W\cap f(\partial B)=\varnothing.
\]

Let $\llbracket B\rrbracket$ denote the integral $2$-current defined by
integration over $B$ with its standard orientation. Since $f$ is Lipschitz on
a neighborhood of $\overline B$, its push-forward
$f_\#\llbracket B\rrbracket$ is well defined. We regard its restriction to
test forms supported in $W$ as a current $T$ on $W$; thus, for every smooth
compactly supported $2$-form $\varphi$ on $W$,
\[
 T(\varphi)
 =
 (f_\#\llbracket B\rrbracket)(\varphi)
 =
 \int_B f^*\varphi.
\]
See \cite[Section 4.1.14]{Federer1969} for the Lipschitz push-forward of currents and the identity
$\partial(f_\#S)=f_\#(\partial S)$. Hence, for every smooth compactly supported
$1$-form $\psi$ on $W$,
\[
 \partial T(\psi)
 =
 (f_\#\partial\llbracket B\rrbracket)(\psi)
 =
 \int_{\partial B} f^*\psi
 =0,
\]
because $W\cap f(\partial B)=\varnothing$. Thus $T$ is closed in $W$.
Since $f$ is bi-Lipschitz, $T$ is a locally integral rectifiable
$2$-current. Moreover, by Step~1, for almost every $\lambda\in B$
the plane $df_\lambda(\CC)$ is a complex line. Hence $T$ is of type
$(1,1)$. Indeed, every form of type $(2,0)$ or $(0,2)$ vanishes on a
one-dimensional complex vector space, and therefore its pullback by
$f$ vanishes almost everywhere on $B$.

The support of $T$ has locally finite $\mathcal H^2$-measure, since
$f$ is Lipschitz. In particular,
\[
 \mathcal H^3(\supp T)=0.
\]
The theorem of Harvey--Shiffman
\cite[Theorem~2.1]{HarveyShiffman1974} therefore implies that $T$ is
a holomorphic $1$-chain in $W$. Consequently $\supp T$ is a
one-dimensional complex analytic set in $W$.

Finally,
\[
 f(B')\subset\supp T.
\]
Indeed, let $\lambda\in B'$ and let $V\Subset W$ be any neighborhood
of $f(\lambda)$. Then $f^{-1}(V)$ contains a neighborhood of $\lambda$.
Since $f$ is bi-Lipschitz, $f(B)\cap V$ has positive
$\mathcal H^2$-measure. Hence the current $T$ does not vanish in $V$,
so $f(\lambda)\in\supp T$.

Thus $f(B')$ is contained in a one-dimensional complex analytic set.
Since $\lambda_0$ was arbitrary, the assertion follows.

\medskip

\noindent
\textbf{Step 3.}
\emph{The map $f$ is holomorphic or anti-holomorphic on $U$.}

\medskip

Fix $\lambda_0\in U$. By Step~2, after shrinking to a disc
$D\Subset U$ centered at $\lambda_0$, there is a one-dimensional complex
analytic set $A$ in a neighborhood of $f(D)$ such that $f(D)\subset A$.
The singular set of $A$ is discrete. Since $f$ is injective, after shrinking
$D$ once more we may arrange that either $f(D)$ lies in the regular part of
$A$, or $f(\lambda_0)$ is the only singular point of $A$ met by $f(D)$.

First suppose that $f(D)$ lies in the regular part of $A$. Choose a
holomorphic coordinate $\chi$ on $A$ near $f(\lambda_0)$ and shrink $D$ so that
$\chi$ is defined on $f(D)$. Then
\[
 g:=\chi\circ f:D\longrightarrow\CC
\]
is an injective locally bi-Lipschitz map. By invariance of domain, $g(D)$ is
open and $g:D\to g(D)$ is a homeomorphism. Hence $g$ is either orientation
preserving or orientation reversing.

By Step~1, at almost every point of $D$ the differential $dg$ is either
complex-linear or conjugate-linear. If $g$ preserves orientation, the first
alternative holds almost everywhere, so $g_{\bar z}=0$ almost everywhere. If
$g$ reverses orientation, then $g_z=0$ almost everywhere. Since $g$ is locally
Lipschitz, these identities hold in the sense of distributions, and Weyl's
lemma implies that $g$ is respectively holomorphic or anti-holomorphic. Since
$\chi^{-1}$ is holomorphic, the same is true of $f$ on $D$.

If $f(\lambda_0)$ is singular, apply the preceding argument on the punctured
disc $D\setminus\{\lambda_0\}$. This punctured disc is connected, so the local
holomorphic or anti-holomorphic alternatives agree on overlaps and give one
alternative on all of $D\setminus\{\lambda_0\}$. Since $f$ is continuous at
$\lambda_0$, the removable-singularity theorem extends the same alternative
across $\lambda_0$.

Thus every point of $U$ has a neighborhood on which $f$ is either holomorphic
or anti-holomorphic. These two alternatives cannot change on an overlap,
because $df$ is injective wherever it exists. Since $U$ is connected, one
alternative holds on all of $U$. This proves Theorem~\ref{thm:disc}.

\section{Proof of Theorem~\ref{thm:orbit}}

We prove the equivalence in three steps.

\medskip

\noindent
\textbf{Step 1.}
\emph{Assume \textup{(a)}. Then $f(V)\subset\mathcal O_r$}

\medskip

The infinitesimal equality implies that
$df_z$ has real rank three at every $z\in V$. Hence $f$ cannot
map a relatively open subset of $V$ into the royal variety $\cR$, which has
real dimension two. Thus the points $z$ for which
$f(z)\notin\cR$ are dense in $V$.

Take first such a point $z_0$, write $f(z_0)\in\mathcal O_q$, and
choose automorphisms $\alpha,\beta$ of $\GG$ such that
\[
 \alpha(z_0)=x_r,
 \qquad
 \beta(f(z_0))=x_q.
\]
Then
\[
 L=d\beta_{f(z_0)}\circ df_{z_0}\circ
 d\alpha^{-1}_{x_r}:\CC\times i\RR\longrightarrow\CC^2
\]
is real linear and satisfies
\[
 \kappa_{\GG}(x_r;v)=\kappa_{\GG}(x_q;L(v)),
 \qquad v\in\CC\times i\RR.
\]
We determine $L$ directly from formula~\eqref{eq:kappa-max}.

For $0<\nu<1$ put $F_\nu(a,b)=\kappa_{\GG}(x_\nu;(a,b))^2$. If $a\ne0$, write
$b/a=u+iv$ and put $\alpha_\nu=(1-\nu)/2$, $\beta_\nu=(1+\nu)/2$. By
\eqref{eq:ellipse}, up to the positive factor $|a|^2/(1-\nu^2)^2$, the function
$F_\nu$ is
\[
 H_\nu(u,v)=\max_{\theta\in\RR}
 \bigl((u-\alpha_\nu\cos\theta)^2+(v+\beta_\nu\sin\theta)^2\bigr).
\]
If $v\ne0$, choosing the sign of $\sin\theta$ which maximizes the expression
reduces the problem, with $c=\cos\theta$, to maximizing
\[
 u^2+v^2+\beta_\nu^2-\nu c^2-2\alpha_\nu u c
 +2\beta_\nu|v|\sqrt{1-c^2}
\]
on $[-1,1]$. Its second derivative on $(-1,1)$ is
\[
 -2\nu-\frac{2\beta_\nu|v|}{(1-c^2)^{3/2}}<0.
\]
Thus the maximizer is unique and nondegenerate. If $v=0$, one obtains
\[
 H_\nu(u,0)=
 \begin{cases}
 \displaystyle \beta_\nu^2\left(1+\frac{u^2}{\nu}\right),
 & |u|\le \dfrac{2\nu}{1-\nu},\\[5pt]
 \displaystyle (|u|+\alpha_\nu)^2,
 & |u|\ge \dfrac{2\nu}{1-\nu}.
 \end{cases}
\]
Consequently the non-real-analytic locus of $F_\nu$, away from the origin, is
\begin{equation}\label{eq:Sigma}
\Sigma_\nu
=
\{(0,b):b\ne0\}
\cup
\left\{
(a,b):a\ne0,
\ \Im(b\overline a)=0,
\ |b|\le\frac{2\nu}{1-\nu}|a|
\right\}.
\end{equation}
Moreover, for $z=x+iy$ with $x,y\in\RR$,
\begin{equation}\label{eq:vertical-first-variation}
(1-\nu^2)^2
F_\nu(\varepsilon z,i(1+\varepsilon\tau))
=
1+\varepsilon
\left(
2\tau+
\sqrt{(1+\nu)^2x^2+(1-\nu)^2y^2}
\right)
+O(\varepsilon^2),
\end{equation}
and, if $s,t\in\RR$, $s\ne0$ and $|t|<2\nu|s|/(1-\nu)$,
\begin{equation}\label{eq:wedge-quadratic}
F_\nu(is,it)
=
\frac1{4(1-\nu)^2}
\left(s^2+\frac{t^2}{\nu}\right).
\end{equation}
At $|t|=2\nu|s|/(1-\nu)$ the two analytic branches meet.

Put $E=\CC\times i\RR$ and $W=L(E)$. Since
$F_r=(F_q|_W)\circ L$, the non-real-analytic locus is preserved by $L$.
If the one-sided directional derivative of a function $G$ at $\xi$ exists,
put
\[
 D_\xi G(h)=\lim_{\varepsilon\downarrow0}
 \frac{G(\xi+\varepsilon h)-G(\xi)}{\varepsilon},
 \qquad
 N_\xi G(h)=\frac12\bigl(D_\xi G(h)+D_\xi G(-h)\bigr).
\]
At $\xi=(0,i)$, formula~\eqref{eq:vertical-first-variation} gives, up to a
positive constant,
\[
 N_\xi F_r(z,i\tau)=
 \sqrt{(1+r)^2(\Re z)^2+(1-r)^2(\Im z)^2},
\]
whose kernel is exactly the vertical line. At a nonvertical point of
$\Sigma_q$, the preceding maximization shows that, after restriction to the
real three-plane $W$, the first variation is either linear or the maximum of
two linear forms. Its even part has kernel of real dimension at least two.
Therefore
\[
 L(\{0\}\times i\RR)\subset\{0\}\times\CC.
\]
After multiplying both target components by a common unimodular constant we
may write
\[
 L(0,it)=(0,idt),\qquad d>0,
\]
and~\eqref{eq:axes} gives $d=\frac{1-q^2}{1-r^2}$.
Since every nonzero point of $\CC\times\{0\}$ belongs to $\Sigma_r$, the
real two-plane $L(\CC\times\{0\})$ is contained in $\Sigma_q$ and has no
nonzero vertical vector. Hence
\[
 L(a,0)=(A(a),\ell A(a)),
\]
where $A:\CC\to\CC$ is a real-linear isomorphism and
$\ell\in\RR$, $|\ell|\le2q/(1-q)$. For small $t$ the vector
$(is,it)$ belongs to $\Sigma_r$, and therefore
\[
 A(i)=i\gamma,
 \qquad
 L(is,it)=\bigl(i\gamma s,i(\ell\gamma s+dt)\bigr)
\]
for some nonzero $\gamma\in\RR$. The equality
$|\ell|=2q/(1-q)$ is impossible, because then
$t\mapsto L(i,it)$ crosses the branch transition at $t=0$, whereas
$t\mapsto F_r(i,it)$ is real analytic there by
\eqref{eq:wedge-quadratic}. Thus $|\ell|<2q/(1-q)$, and on a sufficiently
small cone formula~\eqref{eq:wedge-quadratic} gives
\[
\frac1{4(1-r)^2}
\left(s^2+\frac{t^2}{r}\right)
=
\frac1{4(1-q)^2}
\left(
\gamma^2s^2+\frac{(\ell\gamma s+dt)^2}{q}
\right).
\]
The $st$-coefficient gives $\ell=0$. Hence
$L(\CC\times\{0\})=\CC\times\{0\}$.

Write $L(a,0)=(A(a),0)$. By~\eqref{eq:axes},
\[
 |A(a)|=\frac{1-q}{1-r}|a|,
 \qquad
 d=\frac{1-q^2}{1-r^2}.
\]
Thus $A=cO$, where $c=(1-q)/(1-r)$ and $O$ is real orthogonal. Since
$A(i)$ is purely imaginary,
\[
 A(a)=\pm ca\quad\hbox{or}\quad A(a)=\pm c\overline a.
\]
Comparing the $t^2$-coefficients in the quadratic identity gives
\[
 q(1+r)^2=r(1+q)^2,
\]
hence $(q-r)(1-rq)=0$. Therefore $q=r$, and $c=d=1$. Undoing the
unimodular normalization, we obtain
\[
 L(z,it)=e^{i\theta}(z,\varepsilon it)
 \quad\hbox{or}\quad
 L(z,it)=e^{i\theta}(\overline z,\varepsilon it)
\]
for some $\theta\in\RR$ and $\varepsilon\in\{1,-1\}$.

Thus $f(z)\in\mathcal O_r$ whenever $f(z)\notin\cR$. Such points
are dense in $V$, and $\mathcal O_r$ is relatively closed in $\GG$.
Consequently
\[
 f(V)\subset\mathcal O_r.
\]

\medskip

\noindent
\textbf{Step 2.}
\emph{Condition \textup{(a)} implies condition \textup{(c)}.}

\medskip

Since $f(V)\subset\mathcal O_r$ and $df_z$ is an isomorphism of real
three-dimensional tangent spaces, $f$ is a local $C^1$ diffeomorphism of
$\mathcal O_r$ and hence is locally bi-Lipschitz.

At $x_r$ define
\[
 g_r^0((a,it),(c,is))=\Re(a\overline c)+ts.
\]
The tangent classification in Step~1 shows that the differential of every
automorphism fixing $x_r$ is orthogonal for $g_r^0$. If
$\gamma(x_r)=x\in\mathcal O_r$, set
\[
 g_r|_x(u,v)=g_r^0\bigl(d(\gamma^{-1})_xu,d(\gamma^{-1})_xv\bigr).
\]
This is independent of the choice of $\gamma$ and defines a smooth
Riemannian metric on $\mathcal O_r$. Again by Step~1, every real-linear
isometry between the restricted Kobayashi--Royden tangent norms preserves the
corresponding inner products $g_r$. Hence $f^*g_r=g_r$. The
Myers--Steenrod theorem, in the form of
\cite[Theorem~2.2]{MatveevTroyanov2017}, implies that $f$ is smooth and is a
local Riemannian isometry.

Fix now $z_0\in V$. Choose an automorphism $A_0$ sending
$f(z_0)$ to $z_0$ and an automorphism $C$ sending
$z_0$ to $x_r$, and put
\[
 h=C\circ A_0\circ f\circ C^{-1}.
\]
Then $h(x_r)=x_r$. By Step~1,
\[
 dh_{x_r}(v_1,v_2)=
 \bigl(\varepsilon_1\Re v_1+i\varepsilon_2\Im v_1,
       \varepsilon_3v_2\bigr),
 \qquad \varepsilon_j\in\{1,-1\}.
\]
We only have to determine which sign patterns can occur.

Write $s=x+iy$ and $p=r+u+it$. Formula~\eqref{eq:Delta} gives
$\mathcal O_r$ near $x_r$ as a graph
$u=\varphi(x,y,t)$ with $\varphi(0)=0$ and $d\varphi(0)=0$. Thus
$T_{x_r}\mathcal O_r=\CC\times i\RR$. If
\[
 H_x=T_x\mathcal O_r\cap iT_x\mathcal O_r,
\]
then on this graph
\[
 H=\ker\theta,
 \qquad
 \theta=dt-\varphi_x\,dy+\varphi_y\,dx+\varphi_t\,d\varphi.
\]
Putting $t=0$ in the level equation and comparing quadratic terms gives
\[
 \varphi(x,y,0)=
 \frac{1-r}{4(1+r)}x^2-
 \frac{1+r}{4(1-r)}y^2+o(x^2+y^2).
\]
Hence
\[
 \theta_{x_r}=dt,
 \qquad
 (d\theta)_{x_r}=\frac{2r}{1-r^2}\,dx\wedge dy.
\]
At every nearby point the tangent classification from Step~1 shows that
$dh$ preserves $H$. Therefore $h^*\theta=\mu\theta$ for a nonvanishing smooth
function $\mu$, and $\mu(x_r)=\varepsilon_3$. Differentiating and restricting
to $H_{x_r}$ gives
\[
 h^*(d\theta)_{x_r}=\varepsilon_3(d\theta)_{x_r}.
\]
Since $(d\theta)_{x_r}$ is a nonzero multiple of $dx\wedge dy$, it follows
that
\[
 \varepsilon_1\varepsilon_2\varepsilon_3=1.
\]
The four possible sign patterns are realized by
\[
 (s,p)\mapsto(s,p),\qquad
 (s,p)\mapsto(-s,p),\qquad
 (s,p)\mapsto(\overline s,\overline p),\qquad
 (s,p)\mapsto(-\overline s,\overline p).
\]
Thus there is an automorphism or anti-automorphism $B$ of $\GG$, fixing
$x_r$, such that
\[
 dB_{x_r}=dh_{x_r}
 \quad\hbox{on }T_{x_r}\mathcal O_r.
\]
Then $B^{-1}\circ h$ is a local $g_r$-isometry fixing $x_r$ with identity
differential. A local Riemannian isometry is determined by its value and its
differential at one point, for instance by the exponential map. Hence
$B^{-1}\circ h$ is the identity near $x_r$.

We have proved that $f$ agrees near $z_0$ with the restriction of a
global automorphism or anti-automorphism $A$ of $\GG$. Fix this $A$ and let
$S\subset V$ be the set of points having a neighborhood on which $f=A$.
Then $S$ is nonempty and open. If $z_j\in S$ and
$z_j\to z\in V$, then
$f(z)=A(z)$ and $df_z=dA_z$. The same first-jet
uniqueness gives $f=A$ near $z$. Hence $S$ is closed. Since $V$ is
connected, $S=V$. This proves \textup{(a)}$\Rightarrow$\textup{(c)}.

\medskip

\noindent
\textbf{Step 3.}
\emph{Condition \textup{(b)} implies condition \textup{(a)}, and
\textup{(c)} implies both \textup{(a)} and \textup{(b)}.}

\medskip

Assume \textup{(b)}. By Lemma~\ref{lem:comparison}, $f$ is locally Euclidean
bi-Lipschitz on the real hypersurface $V$, and hence is differentiable almost
everywhere. Let $z$ be a differentiability point and
$v\in T_z\mathcal O_r$. Choose a $C^1$ curve $\eta$ in
$\mathcal O_r$ with $\eta(0)=z$ and $\eta'(0)=v$. Applying
Lemma~\ref{lem:metric-derivative} to $\eta$ and $f\circ\eta$ gives
\[
 \kappa_{\GG}(f(z);df_z v)
 =\kappa_{\GG}(z;v).
\]
Thus the infinitesimal equality holds almost everywhere, and $df_z$
has rank three almost everywhere. As before, the points at which
$f(z)\notin\cR$ are dense. At every such differentiability point the
tangent calculation of Step~1 applies and shows that
$f(z)\in\mathcal O_r$. By continuity and the relative closedness of
$\mathcal O_r$,
\[
 f(V)\subset\mathcal O_r.
\]

With the Riemannian metric $g_r$ constructed in Step~2, we have
$f^*g_r=g_r$ almost everywhere. The low-regularity Myers--Steenrod theorem
\cite[Theorem~2.2]{MatveevTroyanov2017} now shows that $f$ is smooth and is a
local Riemannian isometry. Consequently the Kobayashi--Royden equality holds
everywhere, so \textup{(a)} holds. By Step~2, \textup{(c)} follows.

Finally, every automorphism and anti-automorphism of $\GG$ preserves both the
Kobayashi distance and the Kobayashi--Royden metric. Hence
\textup{(c)} implies both \textup{(a)} and \textup{(b)}. This proves
Theorem~\ref{thm:orbit}.

\end{document}